\documentclass[12pt, reqno, psamsfonts]{article}
\usepackage{amsmath}
\usepackage{amsthm}
\usepackage{amssymb}
\usepackage{amscd}
\usepackage{amsfonts}
\usepackage{amsbsy}
\usepackage{epsfig}
\theoremstyle{plain}
\newtheorem{thm}{Theorem}
\newtheorem{cor}{Corollary}
\newtheorem{lemma}{Lemma}

\newtheorem{example}{Example}
\theoremstyle{definition}

\newcommand{\R}{\mathbb{R}}

\renewcommand{\O}{\mathcal{O}}

\newcommand{\C}{\mathbb{C}}

\def\T{\mathcal T}

\newcommand{\comSSSS}[1]{\marginpar{\tiny QQQ}}

\begin{document}
\title {Rational maps with real multipliers}
\author{Alexandre Eremenko\thanks{Supported by NSF grant DMS-0555279.}
$\;$ and
Sebastian van Strien\thanks{Supported by a Royal Society Leverhulme
Trust Senior Research Fellowship.}} 
\date{\today}
\maketitle

A simple argument of Fatou \cite[Section 46]{Fatou3} shows
that if the Julia set of a rational
function is a smooth curve then all periodic orbits
on the Julia set have real multipliers, see also
\cite[Cor. 8.1]{Milnor}. This argument gives
the same conclusion 
if one only assumes that the Julia set
is {\em contained}
in a smooth curve. By a smooth curve we mean
a curve that has a tangent at every point.

We prove the converse statement:

\begin{thm}\label{thm:main}
Let $f\colon \bar \C\to \bar \C$ be a rational map 
such that the multiplier of each repelling periodic orbit
is real.
Then either the Julia set $J(f)$ is contained in a circle
or $f$ is a Latt\`es map.
\end{thm}

\begin{cor}
If the Julia set of a rational function is contained in a  
smooth curve  then it is contained in a circle. 
\end{cor}

In fact, Theorem~\ref{thm:main} holds  if all repelling periodic points
on some relatively open subset of $J(f)$ have real multipliers.
It follows that even if a relatively open set of the Julia set is contained in a smooth curve,
then the Julia set is contained in a circle.

This corollary generalizes the result of Fatou \cite[Section 43]{Fatou3}
that whenever the Julia set of a rational function is a smooth curve,
this curve has to be a circle or an arc of a circle.
For another proof of the corollary, independent of Theorem~1,
see \cite{BE}. 
We give a more precise description of the maps which can occur:

\begin{thm}
Let  $f\colon \bar \C\to \bar \C$ be a rational map 
whose Julia set $J(f)$ is contained in a circle $C$.
Then there are the following possibilities:

\noindent
(i) $C$ is completely invariant, in which case $f^2$ is a Blaschke
product (that is both components of
the complement of $C$ are invariant under $f^2$). The Julia set
is either $C$ of a Cantor subset of $C$.

If (i) does not hold, then there is a critical point on $C$,
and a fixed point $x_0\in C$ whose multiplier
satisfies $\lambda\in[-1,1]$.
Let $I$ be the smallest closed
arc on $C$ which contains $J(f)$
and whose interior does not contain $x_0$.
Then one of the following holds:

\noindent
(ii) $I$ is a proper arc which is completely
invariant, and $J(f)=I$.

\noindent
(iii) $f(I)$ strictly contains $I$. 
The Julia set is a Cantor subset of $I$ in this case. 

\noindent 
In Cases (ii) and (iii), for each critical point $x$
of $f$ in $I$ 
there exists $N\ge 1$ such that
$f^N(x)\notin \mbox{interior}\,(I)$
(where in Case (ii) $N=1$). All critical points on $J(f)$
are pre-periodic.
\end{thm}

\noindent
{\em Remarks.}  If $f$ is a Blaschke
product preserving a 
circle $C$, then $f:C\to C$ is a covering, so all three
cases (i), (ii) and (iii)
are disjoint. Chebyshev polynomials belong to the
case (ii), and every polynomial satisfying (ii) is
conjugate to $\pm T$, where $T$ is a Chebyshev polynomial.
If $f$ in Case (iii) is a polynomial
one can always take $N=1$,
but there are rational functions satisfying (iii) for which $N>1$,
see Example 3 at the end of the paper.

There are functions $f$ satisfying (ii) which are not
conjugate to polynomials.
A parametric description of functions satisfying (ii)
can be obtained using
\cite[Section 25]{Fatou1}\footnote{We use this opportunity
to notice that the statement of these results of Fatou in the
survey \cite{EL} is wrong. See Fatou's original paper
for the correct statements.}. 

Each $f$ satisfying (ii) is conjugate
to 
$B^2(\sqrt{z})$, where $B$ is an odd rational function
which leaves both upper and lower half-planes invariant,
and whose Julia set equals $\bar\R$.
In the opposite direction, if $B$ is an odd rational function
which leaves both upper and lower half-planes invariant,
and whose Julia set equals $\bar\R$ then
$B^2(\sqrt{z})$ is a rational function whose Julia set
is the ray $[0,\infty]$, and this function satisfies (ii).  


In Cases (ii) and  (iii) the interval $I$ is equal to $C\setminus B_0$ 
where $B_0$ denotes the immediate basin of $x_0$
for the restriction of $f$ on $C$.
In Case (iii)
there exist finitely many closed arcs on C such that
the full preimage of their union is contained in this union. To prove this claim,
take (within $C$) the preimages up to order $N$ of $B_0$,
where
$N$ is as in Theorem~2 and therefore the
union $K$ of the closures 
of these intervals contains all critical values in $I$.
Hence $I\setminus K$ has the following properties: 
$I\setminus K$ contains the Julia
set. As $K$ is forward invariant, $\bar\C$ is
 backward invariant.
As every point of $J(f)$ (and therefore every point
of $I\backslash K$ has all preimages in the closure
of $I\backslash K$, and there are no critical
values in the interior of $I\backslash K$,
we conclude that the closure of $I\backslash K$
is backward invariant.

Theorem 1 is proved in Section 1. In Section 2, we
prove Theorem~2 and discuss rational
functions satisfying (iii) of Theorem 2.

\section{Proof of Theorem 1}

There are only finitely many repelling cycles
which belong to the forward orbits of critical points.
So there exists
a repelling periodic point $p$ of $f$ of period $N$
which does not lie on the forward orbit of a critical point.
Replacing $f$ by $f^N$ we may assume that $N=1$.

Let $\Psi\colon \C\to \bar \C$ be a
holomorphic map which globally linearizes $f$ at $p$, \, 
i.e., 
$$\Psi\lambda=f\Psi,\quad \Psi(0)=p,\quad D\Psi(0)\neq 0. $$
Here $\lambda$ is the multiplier of $p$, so
$\lambda:=Df(p)$ is real. Such a map $\Psi$, which is also called
a Poincar\'e
function \cite{Valiron}, always exists. It is uniquely
defined by the value $D\Psi(0)
\in\C\backslash \{0\}$
which can be prescribed arbitrarily. 
%

\begin{lemma} \label{lem:univ}
If $z\in \Psi^{-1}(p)$ and $p$ is not an
iterate of a critical point then $D\Psi(z)\ne 0$.
\end{lemma}
\begin{proof}
Since $\Psi(z)=
f^{n}\Psi\lambda^{-n}z$, and $\Psi$
is univalent in a neighborhood of $0$,
the result follows.
\end{proof}

We will use several times the following result of 
Ritt \cite{Ritt}:

\begin{lemma}\label{ritt}
The Poincar\'e function is  periodic 
if and only if $f$ is a Latt\`es map, or conjugate to
$\pm T_n$, where $T_n$ is a Chebyshev polynomial, or conjugate
to $z^{\pm d}$.
\end{lemma}

Rational functions with periodic Poincar\'e functions
described in Lemma~\ref{ritt} will be called {\em exceptional}.

\subsection{The linearizing map restricted to certain lines}

A simple curve $\gamma:(0,1)\to\bar \C$ passing through a repelling periodic
periodic point $p$ of period $N$
is called an {\em unstable manifold} for $p$
if there exists a subarc $\gamma_*$
of $\gamma$ containing $p$
so that  $f^{N}$ maps $\gamma_*$
diffeomorphically onto $\gamma$.
Similarly, we say that $\gamma$ is
an {\em invariant curve for} 
$p_{-m}\in f^{-m}(p)$ if $\gamma$ is contained
in an unstable manifold for $p$, 
$p_{-m}\in \gamma$ and $f^m(V_m\cap \gamma)
\subset \gamma$
for some neighborhood $V_m$ of $p_{-m}$.

Choose $Q\in \Psi^{-1}(p)\backslash\{0\}$. 
By Lemma~\ref{lem:univ}, $D\Psi(Q)\neq 0$, so there exist
small topological discs $\O_0\ni 0$ and $\O_1\ni Q$ so that
$\Psi|\O_0$ and $\Psi|\O_1$ are univalent
and $\Psi(\O_0)=\Psi(\O_1)$. 
Hence there exists a biholomorphic
map $\T\colon \O_0\to \O_1$ for which
$$\T(0)=Q\mbox{ and } \Psi\circ \T=\Psi \mbox{ restricted to }\O_0.$$  
For convenience we may choose $\O_i$
so that $\lambda^{-1}\O_0\subset \O_0$,
for example we can take a round disc for
$\O_0$.

First we show that the linearizing map $\Psi$ is special
on certain lines. To prove this, we shall use the following

\begin{lemma}\label{lem:periodic}
For each $Q\in \Psi^{-1}(p)\backslash\{0\}$, 
there exists a sequence $z_n\to 0$ so that $\lambda^nz_n\to Q$
and $\Psi(z_n)$ is a repelling periodic point of period $n$.
There exists a neighborhood $V_n\subset \O_0$ of $z_n$ 
so that 
$f^{n}\colon \Psi(V_n)\to \Psi(\O_0)$ is
biholomorphic.
\end{lemma}
\begin{proof}
Take $n$ so large that the closure of $\lambda^{-n}\O_1$ is contained in
$\O_0$. 
Notice that
$$f^{n}|\Psi(\lambda^{-n}\O_1) = (\Psi|\O_1) \circ \lambda^n \circ (\Psi|\lambda^{-n}\O_1)^{-1},$$
and therefore $f^{n}|\Psi(\lambda^{-n}\O_1)$ is
univalent on
$V_n:=\lambda^{-n}\O_1$. Moreover, 
 $$ f^{n} (\Psi(\lambda^{-n}\O_1))=\Psi(\O_1)=\Psi(\O_0),$$
 and thus 
there exists $z_n\in V_n$ 
so that $\Psi(z_n)$ is a fixed point of $f^{n}$.
\end{proof}

\begin{lemma} \label{lem:identify}
Suppose that $f$ is not an exceptional function from Lemma~2.
Let $p$, $Q$ and $\T\colon \O_0\to \O_1$ be as above
and let $L$ be the line through $0$ and $Q$.
Let $\gamma$ be the arc $\Psi(L\cap \O_0)$. 
Then
\begin{enumerate}
\item $\gamma$ is an unstable manifold for the fixed
point $p$;
\item there exists a sequence $z_n\in L$ so that $\Psi(z_n)$ is a periodic point
of period $n$  and $\gamma$
is an unstable manifold for $\Psi(z_n)$;
\item for each $n$ large enough,
$\gamma$ is an invariant curve for 
$\Psi(\lambda^{-n}Q)$
 (which is in the backward orbit of $p$); 
\item $\Psi(\O_0\cap L)=\Psi(\O_1\cap L)$ and so
$\T$ maps $\O_0\cap L$ diffeomorphically onto $\O_1\cap L$;
\item the set
$\{z\in \O_0 : \, D\T(z)\in \R\}$ is a finite union
of real analytic curves,
one of which is $\O_0\cap L$. 
\end{enumerate}
\end{lemma} 
\begin{proof} 
First we prove that
\begin{equation}
\label{r}
D\T(z_n)\in\R,
\end{equation}
for the points $z_n$ from Lemma 3.
Let 
$x_n=\Psi(z_n)$ be the corresponding periodic points
of period $n$.
Then  $\Psi \lambda^n=f^{n}\Psi$ implies
$$
D\Psi(\lambda^n z_n) \lambda^n = Df^{n}(x_n) D\Psi(z_n).
$$
Since $Df^{n}(x_n)$ and $\lambda$ are real,
 it follows that $D\Psi(\lambda^n z_n)/D\Psi(z_n)\in \R$.
We note that
$\Psi(\lambda^{n}z_n)=f^{n}(\Psi(z_n))
=f^{n}(x_n)=x_n=\Psi(z_n)$
and $z_n\to 0$ and $\lambda^{n}z_n\to Q$ and therefore
\begin{equation}
\label{23}
\T z_n=\lambda^{n}z_n.
\end{equation}
Hence
$D\Psi(\T z_n)/D\Psi(z_n)\in \R$. This implies (\ref{r}).

If $D\T$ is constant on $\O_0$ then $\T$ is
an affine map,
$\T(z)=az+Q$ where $a=D\T(0)\in\R\backslash\{0\}$. 
The identity $\Psi\circ\T=\Psi$ with a non-constant
meromorphic function $\Psi$ implies that $a=\pm 1$
and we conclude that $\Psi$ is periodic. Then $f$
is an exceptional function from Lemma~2,
contrary to our assumption.

From now on we assume that $D\T$ is not constant.
Then the set $X=\{ z\in \O_0: D\T\in\R\}$
is a finite union  of real analytic
curves. 
We are going to prove that $L\cap\O_0$ is one of
these curves.

Without loss of generality we may 
assume that $L$ is the real line. 

Let $\beta$ be a curve in $X$
that contains infinitely many points $z_n$.
As $\lambda^nz_n\to Q$, we conclude $\arg z_n\to 0$,
so $\beta$ is tangent
to $L$ at $0$.  Let $K>0$ be the order of contact of $\beta$
and $L$ at $0$. Since $D\T(0)\in\R\backslash\{0\}$,
the curves $\T\beta$ and $L$ have the same order
of contact at $Q$. Since $z_n\in\beta$ is of the form
$t_n+i(\tau t_n^K+o(t_n^K))$, whereas $\T z_n\in\T\beta$
is of the form $Q+at_n+o(t_n)+i(a\tau t_n^K+o(t_n^K))$,
where $\tau\neq 0$, $t_n\in\R,$ $a=D\T(0)\in\R\backslash\{0\}$
and $Q\in\R\backslash\{0\}$ (remember that we 
assumed $L=\R$).
We have $\arg z_n=\arg\T z_n$ in view of (\ref{23}), so
we obtain for large $n$:
$$(a\tau t_n^K+o(t_n^K))/(Q+at_n+o(t_n))=
(\tau t_n^K+o(t_n^K))/t_n$$
Since $t_n\to 0$,
this is only possible if $\tau=0$.
It follows that $z_n\in L$ for all large $n$. 

As $\beta$ and $L$ intersect at infinitely
many points $z_n$,  we conclude that $\beta=L\cap\O_0$, and this
proves property 5 of the lemma.

Now $\T:\O_0\to\O_1$ is biholomorphic,
$\T(0)=Q$ and $D\T(z)$ is real for real $z$.
This implies property 4.

We put $\gamma=\Psi(L\cap\O_0)$.
Then property 1 is evident:
take $\gamma_*=\Psi(\lambda^{-1}(L\cap\O_0)).$
Property 2 follows from 
$\lambda^n(V_n\cap L)=\O_1\cap L$ (notation from Lemma~3),
and the fact that $f^n:\Psi(V_n)\to\Psi(\O_0)$
is biholomorphic. 
This also implies property~3 because
$\lambda^{-n}Q\in V_n$.
\end{proof}

\subsection{The case that $\Psi^{-1}(J(f))$ is not contained in a line}
We will use following notation:
if $\gamma$ is a curve through $x$ then $T_x\gamma$ will denote
its tangent line at $x$. 

\begin{lemma}\label{lem:lates}
Assume that $\Psi^{-1}(J(f))\not\subset L$.
Then $\Psi$ is a periodic function, and $f$
is a Latt\`es map.
\end{lemma}
\begin{proof}
Throughout the proof, $Q,$ $\Psi$ and $\T\colon \O_0\to\O_1$
will
be as defined before
Lemma~\ref{lem:periodic} (with $\O_1$ a neighborhood
of $Q$).

Since $\Psi^{-1}(J(f))\not\subset L$, there exists $Q^1\in \C\setminus L$ 
so that $\Psi(Q^1)$ is in the backward orbit of $p$,
say $f^{m}(\Psi(Q^1))=p$.
Define  $Q'=\lambda^mQ^1$, then 
$$\Psi(Q')=\Psi(\lambda^mQ^1)=f^m\Psi(Q^1)=p,$$
and thus Lemma~4 applies to the
line $L'$ through $0$ and $Q'$.

As we assume that $\Psi^{-1}(J(f))\notin L$,
there is an infinite set of lines $L'$ as above.
Indeed, it is easy to see that whenever $\Psi^{-1}(J(f))$
is contained in a finite union of lines then it is actually
contained in one line.
We are going to prove that $D\T(z)\in\R$ for
each such line $L'$, thus concluding from part 5 of
Lemma~4 that $f$ is one of the functions
listed in Lemma~2.
For all those functions, except Latt\`es maps,
$\Psi(J(f))$ is a line, so we will conclude that $f$
is a Latt\`es map. 

We denote by $\O^\prime_0$ and $\O^\prime_1$
neighborhoods of $0$ and $Q'$, respectively,
such that $\Psi$ is univalent
in these neighborhoods and
$\Psi(\O^\prime_0)=\Psi(O^\prime_1)$. We choose
a round disc as $\O^\prime_0$.

Applying Lemmas 3 and 4 to the point $Q'$ we obtain
a sequence $z_k\in L'$, $z_k\to 0$ such that
$\lambda^kz_k\to Q'$ and $x_k=\Psi(z_k)$ are repelling
periodic points
of period $k$.
Fix such a point $z=z_k\in\O_0^\prime$, so that
$\lambda^kz\in\O_1^\prime$, and $x=\Psi(z)$ is
the corresponding periodic point (of period $k$),
which does not belong to the forward orbit of a critical
point.

By statements 1, 2 and 4 of Lemma~4, 
$\gamma':=\Psi(L'\cap \O_0')=\Psi(L'\cap \O_1')$ 
is an unstable manifold
for $p$ and also an unstable manifold for $x$.
Since $\Psi|\O_0^\prime$ is univalent, 
the curve $\gamma'$ is smooth and has no self-intersections.
(We chose $\O_0'$ to be a disc,
and so $\gamma'$ is connected.)
Since   $\gamma'$ is an unstable manifold of $x$,
there exists a curve
$\gamma'_*\subset \gamma'$ through $x$ so
that $f^k$ maps $\gamma'_*$ diffeomorphically onto $\gamma'$.
That is, there exists a nested
sequence of curves
$\gamma'_{i,*}\supset \gamma_{i+1,*}'\ni p$
shrinking in diameter to $0$ (with $\gamma_{0,*}'=\gamma'$)
so that $f^k$ maps $\gamma_{i+1,*}'$
diffeomorphically onto $\gamma_{i,*}'$.


Now also consider the linearization $\hat\Psi$
of $f^k$ associated
to the periodic point $x$, i.e. 
\begin{equation}
\label{k}
f^k\hat \Psi=\hat \Psi \mu\quad\mbox{where}
\quad \mu=Df^k(x).
\end{equation}
Let $\gamma_{i,*}$ be the arcs defined a few
lines above, and 
take $i$ so large that there exists a curve
$\hat L_{i}'$ containing $0$
which is mapped by $\hat \Psi$ diffeomorphically
onto $\gamma_{i,*}$. 
Note that $f^{ik}\colon \gamma_{i,*}\to \gamma'$
can be written as 
$\hat \Psi\circ \mu^i \circ (\hat \Psi|\hat L^\prime_{i})^{-1}$
and, since this map is a diffeomorphism onto,
it follows that $\hat \Psi$ is also a
diffeomorphism 
restricted to the curve
$\hat L':=\mu^{i}\hat L^\prime_{i}$, and that
$\hat \Psi(\hat L')=\gamma'$.
In particular there exists $\hat w\in \hat L'$ so that 
$\hat \Psi(\hat w)=p$.

Since $\gamma'$ is an unstable manifold for $x$, 
the curve $\hat L'$ is invariant under $z\mapsto \mu z$.
As the only smooth curve through $0$
which is invariant under real multiplication is a line,
$\hat L'$ must be contained in a line $\hat M$ through $0$. 

Let  $z'=\T(z)\in \O_1$.  For $j\ge 0$ large, 
$w_j:=\lambda^{-jk}(z')$ is contained in $\O_0$. 
Note that $\Psi(w_j)$ tends to
$p=\hat \Psi(\hat w)$ 
as $j\to\infty$. Since
$\hat\Psi$ is a diffeomorphism restricted to $\hat L'$,
and $\hat w\in\hat L'$, 
then for $j$ large enough there exist unique
$\hat w_j$ such that
$\hat\Psi(\hat w_j)=\Psi(w_j).$

Note that 
$$\hat\Psi(\mu^j\hat w_j)=f^{jk}\hat\Psi(\hat w_j)=
f^{jk}\Psi(w_j)=\Psi(\lambda^{jk} w_j)=
\Psi(z')=\Psi(z)=x.$$  
Let $\hat M_j'$ be the line through $0$ and $\mu^j\hat w_j$
and let $\hat M_j\subset \hat M_j'$ be an open line
segment containing the line segment $[\hat w_j,0]$
and contained in a small neighborhood of $[\hat w_j,0]$.
By Lemma 1, there exist neighborhoods $\hat \O_0\ni 0$ and
$\hat \O_1\ni \mu^j\hat w_j$ on each of
which $\hat \Psi$ is biholomorphic,
and $\hat \Psi(\hat \O_0)=\hat \Psi(\hat \O_1)$.
Next  apply 
Lemma~4 to the map $\hat \Psi$
(taking instead of $L,Q$ the line
$\hat M_j'$ and $\mu^j \hat w_j\in \hat \Psi^{-1}(x)$).
This gives that
$\hat \Psi(\hat M_j'\cap \hat \O_0)$ is an invariant manifold
for $x$ and that  $\hat \Psi(\hat M_j'\cap \hat \O_0)=
\hat \Psi(\hat M_j'\cap \hat \O_1)$.
By statements 3 and 4 of Lemma~4,
there exist small neighborhoods $\hat V_j$ of $\hat w_j$,
 $\hat V_j^1$ of $\mu^j\hat w_j$ and $\hat V_j^0$ of $0$
 so that
\begin{equation}
f^{jk}(\hat \Psi(\hat M_j'\cap \hat V_j))=
\hat \Psi(\hat M_j'\cap \hat V_j^1)=\hat \Psi(\hat M_j' \cap \hat V_j^0)\subset
\hat \Psi(\hat M_j).
\label{eq:hatdeckcj'}
\end{equation}
The first equality holds in view of
(\ref{k}) since $\mu$ is real.
Since $\hat w_j$ lies close to $\hat w$
and $\hat \Psi$ is a diffeomorphism restricted to $[\hat w,0]$,
\,\, $\hat \Psi(\hat M_j)$ is a smooth curve which lies
close to $\hat \Psi([\hat w,0])$ (which is
the subarc of $\gamma'$ connecting $p$ and $x$ defined by $\Psi([0,z])$).
It follows that there exists a curve
 $M_j\subset \O_0$ through $w_j$  and $z$ so that $\Psi(M_j)=\hat \Psi(\hat M_j)$.
By (\ref{eq:hatdeckcj'}), there exists a small neighborhood
$V_j$ of $w_j$ so that
$$\Psi(\lambda^{jk}(M_j\cap V_j))=f^{jk}(\Psi(M_j\cap V_j))\subset  \Psi(M_j)=\Psi(\T M_j).$$

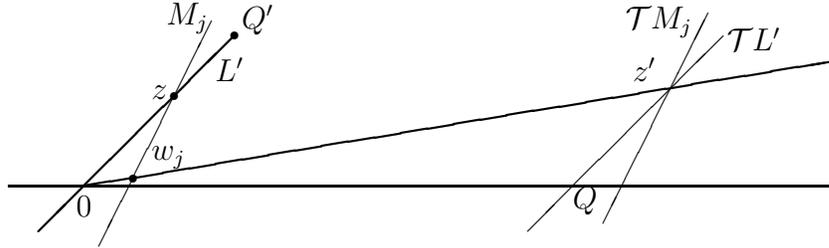
\begin{figure}[h!]
\begin{center}
\unitlength=1mm
\begin{picture}(30,35)(8,-15)
\thicklines
\put(-15,-15){\line(1,0){100}}
\put(-15,-15){\line(-1,0){10}}
\put(-15,-15){\line(6,1){100}}
\put(-15,-15){\line(1,1){20}}
\put(-15,-15){\line(-1,-1){6}}
\put(5,5){\circle*{1}}
\put(-6,-3.5){$z$}
\put(-3,-3){\circle*{1}}
\put(-8.5,-14){\circle*{1}}
\thinlines 
\put(-3,-3){\line(-1,-2){10}}
\put(-3,-3){\line(1,2){5}}
\thinlines
\put(50,-15){\line(1,1){20}}
\put(50,-15){\line(-1,-1){6}}
\put(50,-18){$Q$}
\put(-16,-19){$0$}
\put(2.5,-1){$L'$}
\put(-4,6.5){$M_j$}
\put(57,6){$\T M_j$}
\put(70.5,3){$\T L'$}
\put(6,6){$Q'$}
\put(58,-1){$z'$}
\put(-6,-11.5){$w_j$}
\put(64,0){\line(-1,-2){10}}
\put(64,0){\line(1,2){4}}
\end{picture}
\end{center}
\caption{\label{fig:adjacent_plateaus}
The curves used in the proof
of Lemma~\ref{lem:lates}.
The thinly drawn curves are not necessarily 
line-segments.}
\end{figure}

Since $\lambda^{jk}w_j=z'$, the curves
$\lambda^{jk}M_j$ and $\T M_j$ both go through $z'$ and by the previous
inclusion these curves agree
near $z'$. In particular, 
the tangents of these curves at $z'$ agree:
\begin{equation}
T_{z'}(\lambda^{jk}M_j)=T_{z'}(\T M_j).
\label{eq:deckcj}
\end{equation}
The left hand side of (\ref{eq:deckcj}) is equal to
$T_{w_j}M_j$.
Note that $\hat M_j$ converges to $\hat M$ as $j\to \infty$,
that $\hat \Psi(\hat L')=\gamma'=\Psi(L'\cap \O_0')$ with $\hat L'\subset \hat M$,
and $\Psi(M_j)=\hat \Psi(\hat M_j)$. Hence
$M_j$ converges in the $C^1$ sense to a segment in $L'$.
(Here we use that $\hat\Psi$ is a diffeomorphism on a neighborhood of
$[0,\hat w]$). 

We get therefore that $T_{z'}(\lambda^{jk}M_j)\to T_0L'=T_z L'$
and that the right hand side of (\ref{eq:deckcj}) 
converges to $T_{z'} (\T L')$.
Combined, it follows that
$$T_{z}(L')=T_{z'}(\T L')$$
and so  $D\T(z)\in \R$. Since this holds for a 
whole sequence of points
$z=z_k\in L'$ we obtain that
$D\T(z)\in \R$ for all $z\in L'$. 

As there are infinitely many
such lines $L'$, this implies that $D\T$ is constant, thus
$f$ is a Latt\`es map.
\end{proof}

\subsection{Completion of the proof of Theorem 1}

If $f$ is not a Latt\`es map, Lemma~5 implies that
$\Psi^{-1}(J(f))\subset L$. Without loss of generality
we may assume that that $L$ is the real line.

We recall that the order $\rho$ of a
meromorphic function $\Psi$
is defined by the formula
$$\rho=\limsup_{r\to\infty}\frac{\log T(r,f)}{\log r},$$
where $T(t,f)$ is the Nevanlinna characteristic \cite{Nev}.
According to a theorem of Valiron, \cite[\S 51]{Valiron}
the order of a Poincar\'e function $\Psi$
satisfying
the equation
$$\Psi\lambda=f^N\Psi$$
can be found
by the formula
\begin{equation}\label{88}
\rho=N\log\deg f/\log|\lambda|.
\end{equation}
We claim that under the assumption
that $J\subset \Psi(L)$, 
one can always find infinitely many periodic points $p$
such that the orders of the corresponding
functions $\Psi$ will satisfy $\rho\leq 1+\epsilon$,
for any given $\epsilon>0$.

To prove the claim, we consider the measure of maximal
entropy $\mu$ and the characteristic
exponent
$$\chi(z)=\lim_{n\to\infty}\frac{1}{n}\log|(Df^n)(z)|.$$
The reader may consult the survey \cite{EL}
about these notions. 
According to the multiplicative ergodic theorem,
this limit exists a. e. with respect to $\mu$,
and it is equal a. e. to the average characteristic exponent
\begin{equation}\label{7}
\chi:=\int\log|Df(z)|d\mu(z).
\end{equation}
The average characteristic exponent
is related to the Hausdorff dimension $HD(\mu)$
of measure
$\mu$ by the formula
$$\chi=\frac{\log \deg f}{HD(\mu)},$$
proved in \cite{Le}.
As $\mu$ is supported on the Julia set,
and the Julia set is the image of a line
under a meromorphic function, we conclude that 
$HD(\mu)\leq 1$. So 
\begin{equation}
\label{chi}
\chi\geq\log\deg f.
\end{equation}

Now, $\mu$ is a weak limit of atomic probability
measures $\mu_N$ equidistributed over periodic points
of period $N$. Then (\ref{7}) and (\ref{chi}) imply
that there is
infinitely many periodic points $p$ of periods $N$
such that the multipliers $\lambda$ of these points
satisfies $|\lambda|\geq (1-\epsilon)N(\log\deg f)$. 
We
conclude from Valiron's formula (\ref{88})
that the order
of the Poincar\'e function $\Psi$ is at most $(1-\epsilon)^{-1}$,
as advertised.

We may assume without loss of generality
that $\{0,\infty\}\subset J(f)$ and that $p=\Psi(0)\in\R$.
(This can be achieved by conjugating $f$ by a fractional-linear
transformation).
As we also assume that $L=\R$, the zeros $a_j$ and poles
$b_j$ of $\Psi$ are all real. Taking $\epsilon=1/3$ we obtain
a Poincar\'e function of order at most $3/2$.
According to a theorem of Nevanlinna \cite{GO,H,Nev},
our function $\Psi$ of order less than 2 has a canonical
representation
$$\Psi(z)=be^{az}\frac{\prod_j(1-z/a_j)e^{z/a_j}}{
\prod_j(1-z/b_j)e^{z/b_j}}$$
As $b=\Psi(0)$, $a_j$ and $b_j$ are all real,
we conclude
\begin{equation}
\label{12}
\Psi(z)=e^{icz}g(z),
\end{equation}
where the function $g$ is real on the real line,
and the constant $c=\Im a$ is real. 
If $c=0$ then $\Psi(\R)$ is contained in the
real line and this completes the proof.

Suppose now that $c\neq 0$.
We assume as before
that the point $p$ does not belong to the critical
orbit of $f$. Then $p$ is not a critical value of $\Psi$.
Suppose that $\Psi(z_n)=p$ for $n=1,2,\ldots$, then all $z_k$
are real.
Put $\Psi_n(z)=\Psi(z-z_n)$.
Let $U$ be a small interval around zero on the real line
such that $\Psi$ is univalent in $U$.
Let $\gamma_n=\Psi_n(U)$. These are analytic curves,
and (since the Julia set is perfect)
any two of them have infinitely many intersection
points having an accumulation at  $p=\Psi(0)$.
We conclude that all these $\gamma_n$ are reparametrizations
of the same curve:
$\gamma_n=\gamma$. Now each function $\Psi_n$ maps $U$ onto
the same curve $\gamma$,
and (\ref{12}) implies that
the {\em rate of change of the arguments} of $\Psi_n(x)$
is the same non-zero constant $c$.
We conclude that all $\Psi_n$ are
equal which implies that $\Psi$ is a periodic function.
According to Lemma~2 this can happen only if
$f$ is conjugated to $z^d$ or to a Chebyshev polynomial or to
a Latt\`es map. This proves our Theorem in the case
that $c\neq 0$ and thus completes the proof.

\section{Rational functions with real Julia sets}

{\em Proof of Theorem 2.}
Evidently $f(C)\subset C$.
If there are no critical points on $C$,
then the restriction $f:C\to C$ is a covering.
The degree of this covering should be equal to $\deg f$
since every point of the Julia set has $\deg f$ preimages
in $C$. Thus $C$ is completely invariant and $f^2$
is a Blaschke product.
From now on we assume that $f$ has a critical point
on $C$.

If $J(f)=C$ then both components of the complement
of $C$ are invariant under $f^2$,
so $f^2$ is a Blaschke product in this case as well.

If $J(f)\neq C$, the set of normality is
connected thus there is a fixed point
$z_0$ to which the iterates on the set of
normality converge. As $f(C)\subset C$,
$f$ commutes with reflection
with respect to $C$. This implies that $z_0\in C$.
Evidently, the multiplier of $z_0$ satisfies
$-1\leq\lambda\leq 1$.

We may assume without loss of generality
that $C=\bar\R,$ and $z_0=\infty$. Let $I=[a,b]$ be the convex hull
of the Julia set. This means that $\bar\R\backslash I$ is the immediate
basin of attraction of $\infty$ for the restriction $f\vert_{\bar\R}$.
As the boundary of the immediate basin is invariant,
we obtain

\begin{lemma}\label{A}
The set $\{ a,b\}$ is $f$-invariant.
If $f([a,b])\subset [a,b]$ then $J(f)=[a,b]$.

\end{lemma}

If $f([a,b])\not\subset [a,b]$ then there exists an interval
in $(\alpha,\beta)\subset [a,b]$ which is mapped  by $f$ 
outside $[a,b]$ (and $\alpha,\beta$ are mapped into $a$ or $b$.
Since the preimages of $\alpha,\beta$ are dense in the Julia
set, it follows that in this case the Julia set if a Cantor set.

%

%


\begin{lemma}\label{B}
Each critical points of $f$ in $I$ is contained
in the closure of a real interval which is component 
of the basin of $x_0$.
In particular, each critical point in the Julia set 
is pre-periodic. 
\end{lemma}
\begin{proof}
Let us call a point $x_0\in J(f)$ an {\em endpoint} of $J(f)$
if $J(f)$ accumulates
to $x_0$ only from one side (left or right). It is clear
that the endpoints of $J(f)$ are boundary points
of the basin of $x_0$.
On the other hand, if $x$ is a critical
point and $c=f(x)$ the corresponding critical value then
one of the equation $f(x)=c+\epsilon$ or $f(z)=c-\epsilon$
has non-real solutions in a neighborhood of $x$
for all sufficiently small $\epsilon$.
Thus the critical value $c\in J$ has to be an endpoint of
$J(f)$.
\end{proof}

So there exists  for each critical points $x\in (a,b)$
of $f$ an integer $N\ge 1$ with $f^N(x)\notin (a,b).$


This proves Theorem 2.

\subsection{Polynomials with real Julia sets}

For polynomials with real Julia sets, a complete
parametric description is possible.
 
Let $f$ be a polynomial of degree $d$ whose Julia set
$J$ is real. We may assume that the convex hull of $J$ is $[0,1]$.
Then all $d$ zeros of $f$ are real (belong to $[0,1]$).
Thus all critical points are also real and belong to
$[0,1]$. Let $c_1,\ldots,c_{d-1}$ be the critical
values enumerated left to right. Then the condition that the equation
$f(z)=1$ has all solutions real implies that
all $c_j$ are outside the interval $(0,1)$. 
Moreover, we obtain for odd
$d$ that $0$ and $1$ are either fixed or make a $2$-cycle.
For even $d$ we have $f(0)=f(1)\in\{0,1\}$.

Now, critical values of such polynomials satisfy
\begin{equation}\label{11}
(-1)^jc_j
\quad\mbox{
is of constant sign.}
\end{equation}
This solves the classification problem completely.
We can prescribe {\em arbitrarily} $d-1$ critical values $c_j\in\R
\backslash(0,1)$
satisfying (\ref{11}). Then there exists a real polynomial
with these critical values (ordered sequence!).
This polynomial is unique up to the change of the
independent variable $z\mapsto az+b$ with positive $a$ and real $b$.
Using this change of the variable we achieve
that the convex hull of the set $\{ z:f(z)\in\{0,1\}\}$ is
$\{0,1\}$.

Thus there is a bijective correspondence between sequences of
critical values $(c_1,\ldots,c_{d-1})$ satisfying
$c_j\in\R\backslash(0,1)$
and (\ref{11}) and polynomials with the property that the convex hull
of the Julia set is $[0,1]$. Chebyshev polynomials correspond
to the case $c_j\in\{0,1\}$. All other polynomials of our class
have Cantor Julia sets.

\subsection{Rational functions of the class (iii) in Theorem 2}

We were unable to give any classification
of these functions, so we only give several examples.


\begin{example}
{\rm The simplest non-polynomial example of case (iii)
is a perturbation of a quadratic polynomial.
Consider  $f(z)= (z^2-4)/(1+cz)$ with $c\in \R$. 
If $|c|<1$, this map has an attractor at $\infty$ with multiplier $c$.
Note that $f'(z)=\dfrac{cz^2+2z+4c}{(1+cz)^2}$
and this has two real zeros when $-1/2<c<1/2$. 
To compute $f^{-1}(\R)$, we note that $f(z)=w$ 
is equivalent to 
$$z=\dfrac{cw\pm \sqrt{c^2w^2+4w+16}}{2}.$$
It follows that when $|c|>1/2$, $c\in \R$, then $f^{-1}(\R)\subset \R$ and so 
$f$ is a Blaschke product, while for $|c|<1/2$, $c\in \R$
we find $f^{-1}(\R)\not\subset \R$ and so $f$ is not a Blaschke product.
Note that $f$ is a  Blaschke product with 
an attracting fixed point at $\infty$ if $c\in \R$ and $1/2<|c|<1$. 

As remarked, for $|c|<1/2$,  $f$ is not a Blaschke product. Let us determine
its Julia set. There exists an interval
$I=[p,q]$ containing $0$ so that  $f(p)=f(q)=q$, 
$f\colon I \to \R$ is continuous and has a minimum at some $c\in \mbox{int}(I)$
with $f(c)<p$. Hence there exists two disjoint intervals
$I_0,I_1$ in $I$ which are mapped diffeomorphically onto $I$
and so $f$ has a full horseshoe $\Lambda$ in $[p,q]$.
Each other point is in the basin of the attractor at $\infty$.
Since $f$ has degree two, this horseshoe is also backward invariant, $f^{-1}(\Lambda)=\Lambda$  and so  it follows that $J(f)=\Lambda\subset \R$.}
\end{example}

\begin{example}
{\rm A function of the type (iii)
can have a neutral rational fixed point.
Indeed, take
$f(z)=\dfrac{(z-2)(z+c)(z-c)}{(z-1)(z+1)}$
with $c\in (0,1)$ close to $1$. Then
$\infty$ is a parabolic fixed point which attracts
real points $x\in (-\infty,-1)$ and repels points with $x\in \R$
and $x$                
 large. (Indeed, $f(x)<x$ for $x\in \R$ and
$|x|$ large because $\dfrac{(x+c)(x-c)}{(x-1)(x+1)}>1$ for
$|x|>1$ and therefore $f(x)<x$  when $x\in (-\infty,-1)$. A similarly argument shows
that $f(x)<x$ when $x\in (1,\infty)$.) 
The map $f$ has a unique minimum $c\in (-1,1)$ with $f(c)<-1$.
There are three  disjoint intervals
$I_1,I_2,I_3$ with $I_1,I_2\subset (-1,1)$ and
$I_3\subset (1,\infty)$ 
such that $f$ maps each of these diffeomorphically onto $(-1,\infty)$.
So the Julia set contains a set $\Lambda\subset (-1,\infty)$ on which 
$f$ acts as subshift of three symbols.
Since $f$ has degree $3$, it follows that each preimage of this interval
again lies inside this interval. Hence $J(f)=\Lambda\subset \R$.
Each point outside $\Lambda$ is in the basin of $\infty$.
Clearly $f$ is not a Blaschke product (there exist critical points on $\R$
so $f^{-1}(\R)$ is not contained in $\R$).}
\end{example}

Our last example shows that in general one cannot take
$N=1$ in Case (iii) of Theorem 2.

\begin{example}

{\rm We begin with a Blaschke product of degree $2$,
$$g(z)=Kz\frac{z-a}{z-p}, \quad 0<p<a<1,$$
where the constants are chosen such that $K>1,$ $f(1)=1$,
and $f'(1)>1$.
This function has two branches defined on subintervals of $[0,1]$
that map each subinterval on the whole $[0,1]$,
so the Julia set is a Cantor set whose convex hull is $[0,1]$.
Fix a closed interval $I\subset(p,a)$ on which $f(x)\leq-1$,
(i.e. $I$ is in the basin of the attractor at infinity), 
and let $c$ be the middle point of this interval.
Let $b$ be the preimage of the point $c$ on the interval $[a,1]$.
Now we make a small perturbation of $g$, so that the resulting
rational function of degree $3$ is very close to $g$ on
$[0,1]$ minus a small neighborhood of the point $b$.
Our function is
$$f(z)=K(\epsilon)g(z)\frac{z-b+\epsilon}{z-b-\epsilon},$$
where $\epsilon$ is a very small positive number, and $K(\epsilon)$
is chosen so that $f(1)=1,$ so that $K(\epsilon)\to 1$
as $\epsilon\to 0$. It is clear that $f$ 
has two critical
values $c_1<c_2$ on $[0,1]$ at the critical points near $b$.

It is also easy to see that these critical values both tend to $c$
as $\epsilon\to 0$. 
(Indeed, fix $\delta>0$ small and let $V=[b-\delta,b+\delta]$
be a small neighborhood  of $b$. Our function 
$f$ converges to $g$ and also $f'$ converges
to $g'$ outside $V$ (as $\epsilon\to 0$). 
In particular $f(b+\delta)$ is close to $c=g(b)$,
and $f$ is {\em increasing} at this point $b+\delta$.
But $f$ also has a pole at $b+\epsilon<b+\delta$,
and it is {\em decreasing} on the right hand side
of this pole. It follows that $f$ has a critical point
(a minimum) on the interval $[b+\epsilon, b+\delta]$ with
critical value at most $g(b+\delta)+\delta$ (when $\epsilon>0$ is close to zero), 
which is close to $c=g(b)$.
There is also another critical point on the other side of $b$,
where the critical value is greater than $g(b-\delta)-\delta$.
As the right critical value is evidently greater than the left one,
both critical values tend to $c=g(b)$.)

So if $\epsilon$ is small enough,
we have $f([c_1,c_2])\subset I$, thus the whole interval
$[c_1,c_2]$ escapes from $[0,1]$ under the second iterate of $f$,
and we conclude that $J(f)\subset [0,1]$,
because each point of $[0,1]\backslash [c_1,c_2]$
has three preimages in $[0,1]\backslash [c_1,c_2]$.}

\end{example}

{\em Department of Mathematics,

Purdue University

West Lafayette, IN 47907

USA

eremenko@math.purdue.edu}
\vspace{.1in}

{\em Maths. Dept., University of Warwick,

Coventry CV4 7AL,

UK

strien@maths.warwick.ac.uk}


\begin{thebibliography}{1}
\bibitem{BE} W. Bergweiler and A. Eremenko,
Meromorphic functions with linearly distributed values
and Julia sets of rational functions,
arXiv:0808.3800, to appear in Proc. AMS.
\bibitem{EL} A. Eremenko and M. Lyubich, Dynamics of analytic
transformations, Leningrad Math. J., 1 (1990) 563--634.
\bibitem{Fatou1} P. Fatou, Sur les \'equations
fonctionnelles.
Premiere m\'emoire, Bull. Soc. Math. France, 47 (1919) 161--271.
\bibitem{Fatou3} P. Fatou, Sur les \'equations
fonctionnelles.
Troisi\`eme m\'emoire, Bull. Soc. Math. France, 48 (1920) 208--314.
\bibitem{GO} A. Goldberg and I. Ostrovskii,
Value distribution of meromorphic functions, AMS, Providence, RI, 2008.
Troisieme memoire, Bull. Soc. Math. France, 48 (1920) 208--314.
\bibitem{H} W. Hayman, Meromorphic functions, Clarendon Press,
Oxford, 1964.
\bibitem{Le} F. Leddrapier, Quelques propri\'et\'es ergodiques des applications rationnelles, C. R. Acad. Sci., 299 (1984) 37--40.
\bibitem{Milnor} J. Milnor, Dynamics in One Variable,
Princeton Univ. Press, Princeton, NJ, 2006.
\bibitem{Nev} R. Nevanlinna, Analytic functions, Springer,
NY, 1970.
\bibitem{Ritt} J. F. Ritt, Periodic functions with a
multiplication theorem, Trans. Amer. Math. Soc., 23 (1922) 16--25.
\bibitem{Valiron} G. Valiron, Fonctions analytiques,
Presses universitaires de France, Paris, 1954.
\end{thebibliography}
\enddocument